\documentclass{amsart}
\usepackage[margin=1.0in]{geometry}

\usepackage{amsfonts}

\usepackage{setspace}
\usepackage{graphicx}

\usepackage{hyperref}
\usepackage{graphicx}
\usepackage{amsmath, amsthm, amscd, amsfonts, amssymb, graphicx, color}
 \usepackage{url}
\usepackage{enumerate}
 \makeatletter
\let\reftagform@=\tagform@
\def\tagform@#1{\maketag@@@{(\ignorespaces\textcolor{blue}{#1}\unskip\@@italiccorr)}}
\renewcommand{\eqref}[1]{\textup{\reftagform@{\ref{#1}}}}
\makeatother
\usepackage{hyperref}
\hypersetup{colorlinks=true, linkcolor=red, anchorcolor=green,
citecolor=cyan, urlcolor=red, filecolor=magenta, pdftoolbar=true}

\usepackage{epsfig}        
\usepackage{epic,eepic}       

\newtheorem{theorem}{Theorem}
\theoremstyle{plain}

\newtheorem{lemma}{Lemma}

\numberwithin{equation}{section}

\begin{document}
\title[Mercer's inequality for $h$-convex functions]{Mercer's inequality for $h$-convex functions}
\author[M.W. Alomari]{M.W. Alomari}

\address{Department of Mathematics, Faculty of Science and
Information Technology, Irbid National University, 2600 Irbid
21110, Jordan.} \email{mwomath@gmail.com}
\date{\today}


\begin{abstract}
A generalization of Mercer inequality for $h$-convex function is
presented. As application, a weighted  generalization of triangle
inequality is given.
\end{abstract}

\maketitle
\section{Introduction}

The  class of $h$-convex functions, which generalizes convex,
$s$-convex (denoted by $K_s^2$, \cite{B1}), Godunova-Levin
functions (denoted by $Q(I)$, \cite{GL}) and $P$-functions
(denoted by $P(I)$, \cite{PR}), was introduced by Varo\v{s}anec in
\cite{V}. Namely, for real intervals $I$ and $J$,  the $h$-convex
function is defined as a non-negative function $f : I \to
\mathbb{R}$ which satisfies
\begin{align*}
f\left( {t\alpha +\left(1-t\right)\beta} \right)\le
h\left(t\right) f\left( {\alpha} \right)+ h\left(1-t\right)
f\left( {\beta} \right),
\end{align*}
where $h:J \to \mathbb{R}$ is a non-negative function defined on
$J$, such that $t\in   (0, 1)\subseteq J\subseteq
\left(0,\infty\right) $ and $x,y \in I $. Accordingly, some
properties of $h$-convex functions were discussed in the same work
of Varo\v{s}anec. The famous references about these classes are
\cite{B1}--\cite{HL} and \cite{MP}--\cite{PR}.

Let $w_1, w_2,\cdots,w_n$ be positive real numbers $(n \ge 2)$ and
$h:J\to \mathbb{R}$ be a non-negative supermultiplicative
function. In \cite{V}, Varo\v{s}anec discussed the case that, if
 $f$ is a non-negative $h$-convex on $I$, then for $x_1,x_2,\cdots,x_n \in I$ the following inequality
 holds
\begin{align}
f\left( {\frac{1}{{W_n }}\sum\limits_{k = 1}^n {w_k x_k } }
\right) \le \sum\limits_{k = 1}^n {h\left( {\frac{{w_k }}{{W_{n }
}}} \right) f\left( {x_k } \right)}, \label{Var}
\end{align}
where $W_n=\sum\limits_{k = 1}^n{w_k}$. If $h$ is
submultiplicative function and
 $f$ is an $h$-concave then inequality  is
 reversed. In case $h(t)=t$ we refer to the classical version of
 Jensen's inequality.

 If $f$ is convex on $I$, then
 for any finite
positive increasing sequence  $\left(x_k\right)_{k=1}^{n} \in I$,
we have
\begin{align}
f\left( {x_1  + x_n  -   \sum\limits_{k = 1}^n {w_k x_k } }
\right) \le f\left( {x_1 } \right) + f\left( {x_n  } \right) -
\sum\limits_{k = 1}^n {w_k f\left( {x_k } \right)},\label{mercer}
 \end{align}
where  $w_1, w_2,\cdots,w_n$ are positive real numbers such that
$\sum\limits_{k = 1}^n {w_k}=1$. This inequality was established
by Mercer in \cite{Mercer} and it is considered as a variant of
Jensen's inequality.

In this work, a generalization of Mercer inequality for $h$-convex
function is presented. As application, a weighted generalization
of triangle inequality is given.

\section{Mercer analogue inequality for $h$-convex functions}

In order to prove our main result, we need the following Lemma
which generalizes Lemma 1.3 in \cite{Mercer}.
\begin{lemma}
\label{lemma3}Let $h:J \to \mathbb{R}$ be a non-negative
supermultiplicative  function on $J$. Let $\alpha, \beta  \in
[0,1]$ such that $\alpha+\beta =1$ and
  $h\left( \alpha \right) + h\left(
{\beta} \right)\le 1$. For any $h$-convex function $f$ defined on
a real interval $I$ and finite  positive  increasing sequence
$\left(x_k\right)_{k=1}^{n} \in I$, we have
\begin{align}
f\left( {x_1  + x_n  - x_k } \right) \le f\left( {x_1 } \right) +
f\left( {x_n } \right) - f\left( {x_k } \right) \qquad (1\le k\le
n). \label{eq.h.mercer}
\end{align}
 If $h$ is
submultiplicative function, $h\left( \alpha \right) + h\left(
{\beta} \right)\ge 1$ for all $\alpha,\beta\in [0,1]$ with
$\alpha+\beta=1$ and $f$ is an $h$-concave then inequality
\eqref{eq.h.mercer} is
 reversed.
\end{lemma}

\begin{proof}
Let $0<x_1\le \cdots \le  x_{n}$ and $\alpha,\beta   \in [0,1]$
such that $\alpha+\beta =1$ with $h\left( \alpha \right) +h\left(
\beta \right) \le 1$. Following Mercer approach in \cite{Mercer}.
Let us write $y_k=x_1+x_n-x_k$. Then $x_1+x_n=y_k+x_k$, so that
the pairs $x_1, x_n$ and $x_k, y_k$ possess the same midpoint.
Since that is the case there exists $\alpha,\beta \in [0,1]$ such
that $x_k = \alpha x_1  + \beta x_n$ and  $y_k  = \beta x_1  +
\alpha x_n$, where $\alpha+\beta=1$ and $1\le k \le n$. Employing
the $h$-convexity of $f$ we get
\begin{align*}
 f\left( {y_k } \right) = f\left( {\beta x_1  + \alpha x_n } \right) &\le h\left( \beta  \right)f\left( {x_1 } \right) + h\left( \alpha  \right)f\left( {x_n } \right) \\
  &\le \left( {1 - h\left( {\alpha }
\right) } \right)f\left( {x_1 } \right) + \left( {1 - h\left( {\beta} \right) } \right)f\left( {x_n } \right) \\
  &= f\left( {x_1 } \right) + f\left( {x_n } \right) - \left[ {h\left( {\alpha } \right)f\left( {x_1 } \right) + h\left( {\beta} \right)f\left( {x_n } \right)} \right] \\
  &\le f\left( {x_1 } \right) + f\left( {x_n } \right) - f\left( {\alpha  x_1  +  \beta   x_n}\right)\\
  &= f\left( {x_1 } \right) + f\left( {x_n } \right) - f\left( {\alpha  x_1  +  \beta  x_n }\right)\\
  &=f\left( {x_1 } \right) + f\left( {x_n } \right) - f\left(
  {x_k}\right),
 \end{align*}
 and this proves the
required result.
\end{proof}

Now, we are ready to state our main result.
\begin{theorem}
Let $h:J \to  \mathbb{R}$ be a non-negative supermultiplicative
function on $J$.  Let $w_1, w_2,\cdots,w_n$ be positive real
numbers $(n \ge 2)$ such that $W_n=\sum\limits_{k = 1}^n {w_k } $
and  $ \sum\limits_{k = 1}^n h\left( {\frac{{w_k }}{{W_n }}}
\right) \le 1$. If $f$ is $h$-convex on $I$, then for any finite
positive increasing sequence $\left(x_k\right)_{k=1}^{n} \in I$,
we have
\begin{align}
f\left( {x_1  + x_n  -  \frac{1}{{W_n }} \sum\limits_{k = 1}^n
{w_k x_k } } \right) \le f\left( {x_1 } \right) + f\left( {x_n  }
\right) - \sum\limits_{k = 1}^n {h\left( {\frac{{w_k }}{{W_n }}}
\right)f\left( {x_k } \right)}.\label{h.mercer}
 \end{align}
If $h$ is submultiplicative function, $\sum\limits_{k = 1}^n
h\left( {\frac{{w_k }}{{W_n }}} \right) \ge 1$  and $f$ is an
$h$-concave then inequality \eqref{h.mercer} is
 reversed.

\end{theorem}

\begin{proof}
Since $\frac{1}{W_n} \sum\limits_{k = 1}^n {w_k } =1$, we have
\begin{align*}
f\left( {x_1  + x_n  -  \frac{1}{{W_n }} \sum\limits_{k = 1}^n
{w_k x_k } } \right)  &=f\left( {
\sum\limits_{k = 1}^n {\frac{w_k}{W_n }\left(x_1  + x_n  - x_k \right)} } \right)\\
&\le
\sum\limits_{k = 1}^n {h\left( \frac{w_k}{W_n }  \right)f\left(x_1  + x_n  - x_k \right)} \qquad \text{by}\,\, \eqref{Var} \\
  &\le\sum\limits_{k = 1}^n {h\left( \frac{w_k}{W_n }  \right)\left[{f\left(x_1 \right)+f\left(x_n \right)-f\left(x_k \right)}  \right]} \qquad \text{by}\,\, \eqref{eq.h.mercer} \\
  &=\left[f\left(x_1 \right)+f\left(x_n \right)\right]\sum\limits_{k = 1}^n {h\left( \frac{w_k}{W_n }  \right)}-\sum\limits_{k = 1}^n {h\left( \frac{w_k}{W_n }  \right) f\left(x_k \right) } \\
  &\le f\left(x_1 \right)+f\left(x_n \right) -\sum\limits_{k = 1}^n {h\left( \frac{w_k}{W_n }  \right) f\left(x_k \right)
  }\qquad \text{by\, assumption}
   \end{align*}
   and this proves the result in \eqref{h.mercer}.
\end{proof}
  One of the direct application and interesting benefit of \eqref{h.mercer} is to offer an
  upper bound for the converse of $h$-Jensen inequality
  \eqref{Var},
  by rearranging the terms in \eqref{h.mercer} we get
\begin{align}
\sum\limits_{k = 1}^n {h\left( {\frac{{w_k }}{{W_n }}}
\right)f\left( {x_k } \right)} \le f\left( {x_1 } \right) +
f\left( {x_n  } \right) - f\left( {x_1  + x_n  -  \frac{1}{{W_n }}
\sum\limits_{k = 1}^n {w_k x_k } } \right).\label{conv.h.Jensen}
 \end{align}
 For instance, if $f(x)=|x|$, $h(t)=t$  and  $W_n=1$, then we have the
 following refinement of   the celebrated triangle inequality which is of
 great interests itself
\begin{align}
\sum\limits_{k = 1}^n {w_k \left| {x_k } \right|} \le \left| {x_1
} \right| + \left| {x_n  } \right| - \left| {x_1  + x_n  -
  \sum\limits_{k = 1}^n {w_k x_k } } \right|.
 \end{align}
This inequality can be generalized for norms by considering the
mapping $f\left({\bf{x}}\right)=\|{\bf{x}}\|$  $({\bf{x}}\in L)$,
where $L$ is a linear space.


\begin{thebibliography}{5}
\setlength{\itemsep}{3pt}


\bibitem{B1}W.W. Breckner, Stetigkeitsaussagen f\"{u}r eine Klasse
verallgemeinerter konvexer funktionen in topologischen linearen
R\"{a}umen, {\em Publ. Inst. Math.}, {\bf23} (1978), 13--20.




\bibitem{DPP} S.S. Dragomir, J. Pe\v{c}ari\'{c} and L.E. Persson, Some inequalities
of Hadamard type, {\em Soochow J. Math.}, {\bf21} (1995) 335--341.





\bibitem{GL}E.K. Godunova and V.I. Levin, Neravenstva dlja funkcii \v{s}irokogo
klassa, soder\v{z}a\v{s}\v{c}ego vypuklye, monotonnye i nekotorye
drugie vidy funkcii,  Vy\v{c}islitel. Mat. i. Mat. Fiz. Mežvuzov.
Sb. Nau¡c. Trudov, MGPI, Moskva, 1985,   138--142.



\bibitem{HL} H. Hudzik and L. Maligranda, Some remarks on $s$-convex functions,
{\em Aequationes Math.}, {\bf 48} (1994), 100--111.

\bibitem{J}J. Jensen, Sur les fonctions convexes et les in\'{e}galit\'{e}s
entre les valeurs moyennes, {\em Acta Math.}, {\bf30} (1906),
175--193.

\bibitem{Mercer} A.McD. Mercer, A variant of Jensen's inequality, {\em
JIPAM}, {\bf 4} (4) Article 73, 2003.

\bibitem{MP}D.S. Mitrinovi\'{c} and J. Pe\v{c}ari\'{c}, Note on a class of functions
of Godunova and Levin, {\em C. R. Math. Rep. Acad. Sci. Can.},
{\bf12} (1990), 33--36.

\bibitem{MPF}D.S. Mitrinovi\'{c}, J. Pe\v{c}ari\'{c} and A.M. Fink, Classical and New
Inequalities in Analysis, Kluwer Academic, Dordrecht, 1993.








\bibitem{PR} C.E.M. Pearce and A.M. Rubinov, $P$-functions, quasi-convex
functions and Hadamard-type inequalities, {\em J. Math. Anal.
Appl.}, {\bf 240} (1999), 92--104.






\bibitem{V}S. Varo\v{s}anec,  On $h$-convexity, {\em J. Math. Anal. Appl.},  {\bf326}
(2007), 303--11.


\end{thebibliography}
\end{document}